\documentclass{amsart}
\usepackage{lipsum}
\usepackage{graphicx}
\usepackage{amsmath,leftidx,amsthm}
\usepackage[all]{xy}
\usepackage{xspace}
\usepackage{amssymb}
\usepackage[latin1]{inputenc}
\usepackage{graphicx,color}
\usepackage{hyperref,fancyhdr}
\usepackage{fullpage}
\usepackage{versions}
\usepackage{textcomp,listings}
\usepackage{enumitem}
\setenumerate[1]{label=(\arabic*)}

\newtheorem{theorem}{Theorem}[section]
\newtheorem{lemma}[theorem]{Lemma}
\newtheorem{cor}[theorem]{Corollary}

\theoremstyle{definition}

\newtheorem{proposition}[theorem]{Proposition}

\theoremstyle{remark}

\numberwithin{equation}{section}

   \newcommand{\vstrut}{\rule{0in}{.08in}}
    \newcommand{\vstrutb}{\rule{0in}{.12in}}
  \renewcommand{\Re}{\mbox{\rm Re}}
  \renewcommand{\Im}{\mbox{\rm Im}}
    \newcommand{\st}{^{\textstyle {\ast }}}
   \newcommand{\ph}{\mbox{$\varphi$}}
   \newcommand{\tpsi}{\tilde{\psi}}
    \newcommand{\ran}{\mbox{ran}}
    
    \renewcommand{\phi}{\varphi}
   \newcommand{\Ht}{\mbox{$H^{2}$}}
   
   \newcommand{\Hi}{\mbox{$H^{\infty}$}}
   \newcommand{\D}{\mbox{$\mathbb{D}$}}
   \newcommand{\HP}{{\mathbb{H}}}
   \newcommand{\C}{\mbox{$C_{\varphi}$}}
   \newcommand{\CP}{\mathbb{C}}
    \newcommand{\E}{\mathbb{E}}
   \newcommand{\Cs}{\mbox{$C_{\varphi}\st$}}

 	\newfont{\caps}{cmcsc9}  
 	\newfont{\jour}{cmti9}  
\theoremstyle{remark}

\definecolor{green}{rgb}{0,0.5,0}
\definecolor{dkgreen}{rgb}{0,0.6,0}
\definecolor{gray}{rgb}{0.5,0.5,0.5}
\definecolor{mauve}{rgb}{0.58,0,0.82}
\definecolor{purple}{rgb}{0.58,0,0.62}

\renewcommand{\phi}{\varphi}
\renewcommand{\Re}{\mbox{Re\,}}
\renewcommand{\Im}{\mbox{Im\,}}

\def\and{{\quad\text{and}\quad}}

\begin{document}

\title{Posinormal Composition Operators on $\Ht$}
\author{Paul S. Bourdon}
\address{Department of Mathematics\\  University of Virginia\\ Charlottesville, VA 22903}
\email{psb7p@virginia.edu}
\author{Derek Thompson}
\address{Department of Mathematics, Taylor University,  Upland, IN 46989}
\email{theycallmedt@gmail.com}

\begin{abstract}
A bounded linear operator $A$ on a Hilbert space is \textit{posinormal} if there exists a positive operator $P$ such that $AA^{*} = A^{*}PA$.  Posinormality of $A$ is equivalent to the inclusion of the range of $A$ in the range of its adjoint $A^*$.  Every hyponormal operator is posinormal, as is every invertible operator.    We characterize both the posinormal and coposinormal composition operators $C_\phi$ on the Hardy space $H^2$ of the open unit disk $\D$ when $\phi$ is a linear-fractional selfmap of $\D$.  Our work reveals that there are composition operators that are both posinormal and coposinormal  yet have powers that fail to be posinormal.  
\end{abstract}

\maketitle




\section{Introduction}

Let \Ht\ denote the classical \textit{Hardy space}, the Hilbert space of analytic functions $\displaystyle f(z) = \sum_{n=0}^{\infty}a_n z^n$ on the open unit disk \D\ such that $$\|f\|^{2}=\sum_{n=0}^{\infty}|a_n|^{2}<\infty.$$  Every analytic selfmap $\phi$  of $\D$ induces a \textit{composition operator} \C\ on \Ht\ given by $\C f = f \circ \ph$. Littlewood \cite{Lit} proved $C_\phi$ preserves $H^2$; and thus, by the closed-graph theorem, $C_\phi:H^2\rightarrow H^2$ is a bounded linear operator.   The function-theoretic properties of the \textit{symbol} \ph\ of the associated composition operator are deeply connected to the operator-theoretic properties of \C.  We show how the posinormality of $C_\phi$ relates to the function-theoretic properties of $\phi$ and its ``adjoint'' selfmap $\sigma$ of  $\D$.  

Composition operators on \Ht\ have been studied extensively---see \cite{cowen1995composition} and \cite{shapirocomposition}. In particular, composition operators induced by linear-fractional selfmaps of $\D$ constitute a much studied sub-collection of composition operators. One reason for this is that  when $\phi$ is linear fractional, the adjoint operator  $\Cs$ has a useful description as a product of a Toeplitz operator, a composition operator, and an adjoint Toeplitz operator. When $\psi \in \Hi$, the multiplier space of \Ht\ (bounded analytic functions on \D), the \textit{Toeplitz operator} with symbol $\psi$ is simply the multiplication operator defined by $T_{\psi} f = \psi f$. In this paper, $\psi$ always belongs to \Hi, so this definition will suffice; for more on Toeplitz operators,  see \cite[Chapter 7]{douglas1998banach}. Here we present Cowen's adjoint formula:

\begin{theorem}[Cowen's adjoint formula \cite{Cow}]\label{adjformula}  If $\ph(z) = \cfrac{az+b}{cz+d}$ is an analytic selfmap of $\D$, then on $\Ht$, $\Cs = T_g C_{\sigma} T_{h}^{*}$, where $g(z) = 1/(-\overline{b}z+\overline{d})$ and $h(z) = cz + d$ are in $H^\infty$, and $\sigma(z) = (\overline{a}z-\overline{c})/(-\overline{b}z+\overline{d})$, is an analytic selfmap of $\D$.

\end{theorem}

Following MacCluer and Weir \cite{MW}, we call the selfmap $\sigma$ of $\D$ appearing in Cowen's formula for $C_\phi^*$ {\it the adjoint of $\phi$}.    The functions, $g$, $h$, and $\sigma$ of Cowen's adjoint formula arise naturally through the relationship between composition operators on $H^2$ and the reproducing kernels for $H^2$, which are given by $K_\alpha(z) = \frac{1}{1-\bar{\alpha} z}$, $\alpha\in D$.  Letting $\langle \cdot, \cdot \rangle$ denote the inner product of $H^2$, we have $\langle f, K_\alpha \rangle = f(\alpha)$ for every $f \in \Ht$ and $\alpha \in \D$.  For each analytic selfmap $\phi$ of $\D$, we have $C^*K_\alpha = K_{\phi(\alpha)}$.  When $\phi$ is linear fractional, we have 
$$
(C_\phi K_\alpha)(z) = \overline{g(\alpha)}h(z)K_{\sigma(\alpha)}(z),
$$
 where $h, g$ and $\sigma$  are the linear-fractional transformations associated  with Cowen's adjoint formula; in fact, this computation provides a proof of the formula.  
 
  Throughout this paper, we assume all linear-fractional functions are defined on the extended complex plane $\CP_\infty$ (mapping $\CP_\infty$ onto $\CP_\infty$).

Fixed-point theory also plays a large role in the study of composition operators and does so in our study of their posinormality as well. The Denjoy-Wolff theorem guarantees that every analytic selfmap of \D\  (except elliptic automorphisms) will have a unique attracting fixed point $w$ in the closure $\D^-$ of $\D$, and the iterate sequence of \ph\ converges to that point uniformly on compact subsets of \D. We classify the linear-fractional selfmaps of $\D$ as follows:  $\phi$ is of
\begin{itemize}
\item {\it dilation type} if $\phi$ has a fixed point in $\D$,
\item {\it hyperbolic type} if $\phi$ has Denjoy-Wolff point $w$ on $\partial \D$ with $\phi'(w) < 1$, and
\item {\it parabolic type} if $\phi$  has Denjoy-Wolff point $w$ on $\partial \D$ with $\phi'(w) = 1$.
\end{itemize}
A nonconstant selfmap $\phi$ of dilation type will have a second fixed point either on the unit circle or in  the complement of  $\D^-$.  Similarly, a hyperbolic-type selfmap will have a second fixed point in the complement of $\D$ (and will be an automorphism of $\D$ precisely when the second fixed point is on the unit circle).  Finally, for a parabolic selfmap $\phi$, the Denjoy-Wolff point is the unique fixed point of $\phi$ in $\CP_\infty$.

First defined by Rhaly in \cite{Rhaly}, a bounded linear operator $A$ on a Hilbert space is said to be \textit{posinormal} if there exists a positive operator $P$ such that $AA^{*} = A^{*}PA$. The operator $P$ is called an \textit{interrupter} of $A$. Posinormality is a unitary invariant; in fact, as Rhaly points out, if $V$ is an isometry (so that $V^*V = I$) and $A$ is posinormal with interrupter $P$, then $VAV^*$ is posinormal with interrupter $VPV^*$.   If $A^{*}$ is posinormal, then $A$ is \textit{coposinormal}. Based on the work of Douglas in \cite{douglas1966majorization}, Rhaly gives some extremely useful equivalent conditions for posinormality:
 
\begin{theorem}[\cite{Rhaly}]\label{posinormal}

For $A \in \mathcal{B(H)}$, the space of bounded linear operators on a Hilbert space $\mathcal{H}$, the following statements are equivalent:
\begin{enumerate}
    \item[(1)] $A$ is posinormal;
    \item[(2)]  $\ran \textrm{ } A \subseteq \ran \textrm{ } A^*$;
    \item[(3)]  $AA^{*} \leq \lambda^2 A^{*}A $ for some $\lambda \geq 0$; and
    \item[(4)]  there exists $T \in \mathcal{B(H)}$ such that $A = A^{*}T$.
\end{enumerate}
Moreover, if $A$ is posinormal, then there is a unique operator $T$, satisfying 
the properties below, for which $A = A^*T$.  
\begin{itemize}
    \item[(a)] $\| T\|^2 = \inf \{ \mu | AA^{*} \leq \mu A^{*} A \};$
    \item[(b)] $\ker A = \ker T$; and
    \item[(c)] $\ran \textrm{ } T \subseteq (\ran \textrm{ } A)^{-}$. 
\end{itemize}
\end{theorem}

Those who study hyponormal operators should find this result somewhat familiar: e.g.,  every hyponormal operator is posinormal by condition (3) with $\lambda = 1$, and a nonzero operator $A$ is hyponormal  if and only if the operator $T$ of (4) satisfying the uniqueness properties $(a)$ through $(c)$ has norm $1$.   However, many of the key results about hyponormal operators do not hold: for example, the only result known to the authors concerning the spectral behavior of a posinormal operator $A$ is that if $0$ belongs to the point spectrum of $A$, then it also belongs to the point spectrum of $A^*$; in fact, $Av = 0$ implies $A^*v = 0$ when $A$ is posinormal.   Comparatively, different eigenvalues for a hyponormal operator must have orthogonal eigenvectors, and, by Putnam's Theorem \cite{PT},  a hyponormal operator having a spectrum with zero area must be normal. In this paper, we will rely principally on function theory to determine which symbols \ph\ induce posinormal (and coposinormal) composition operators.

The following proposition provides as a corollary our first posinormality result: if $C_\phi:H^2\rightarrow H^2$ is posinormal, then $\phi$ must have a zero in $\D$. 

\begin{proposition}\label{zero}
Suppose \ph\ is a  linear-fractional selfmap of \D. The constant function $1$ is in the range of \Cs\ if and only if $\ph(\beta) = 0$ for some $\beta \in \D$.
\end{proposition}

\begin{proof} Suppose that $\phi(\beta) = 0$ for some $\beta \in \D$.  Then $\Cs K_\beta = K_{\phi(\beta)} = K_0 = 1.$

Conversely suppose that $1$ belongs to the range of $C_\phi^*$.  If $\phi$ is a constant function, taking all $z$ in $\D$ to the number $b$, then by Cowen's adjoint formula (or a  simple direct computation) $C_\phi^* = K_b C_\sigma$ where $\sigma$ is the constant function taking value $0$;  hence, in this case $1$ belongs to the range of $C_\phi^*$ if and only if $b= 0$, so that $\phi$ takes all values in $\D$ to $0$.  

 Now assume $1$ belongs to the range of $C_\phi^*$ and $\phi(z) = (az +b)/(cz+d)$ is nonconstant (equivalently, $ad-bc\ne 0$).  If $b=0$, then $\phi(0) = 0$.    Suppose $b\ne 0$. We show the zero $-b/a$ of $\phi$ must lie in $\D$ to complete the proof.  By Cowen's adjoint formula, we have
$$
 (C_\sigma T_h^* f)(z)= -\bar{b}z + \bar{d}, z\in \D,
$$
for some $f\in \Ht$, where $\sigma(z) = \dfrac{\bar{a}z - \bar{c}}{\vstrut -\bar{b}z + \bar{d}}$ and $h(z) =  cz +d$. Let $q= T_h^* f$, so that the displayed equation above becomes
\begin{equation}\label{1inrange}
 q(\sigma(z)) = -\bar{b}z + \bar{d}, z\in \D,
\end{equation}
where $q$ is an $\Ht$ function.   Let $u = \sigma(z)$.  Then  $z = \dfrac{\bar{c} + \bar{d}u}{\vstrut\bar{a} + \bar{b}u}$ and (\ref{1inrange}) becomes
$$
q(u) = \frac{\overline{ad} - \overline{bc}}{\vstrut\bar{a} + \bar{b}u}
$$
for all $u$ in the range of $\sigma$.  Thus, $q$ and $\psi(u) : = \dfrac{\overline{ad} - \overline{bc}}{\bar{a} + \bar{b}u}$ agree on the  nonempty open subset $\sigma(\D)$ of $\D$.  Because $q$ is analytic on $\D$ and $\psi$ is analytic on $\D$ except possibly at $-\bar{a}/\bar{b}$, the functions $q$ and $\psi$ must agree on $\D$ except possibly at $-\bar{a}/\bar{b}$.   However, $-\bar{a}/\bar{b}$ is a pole of $\psi$  ($\overline{ad} - \overline{bc}$ is nonzero) and thus $-\bar{a}/\bar{b}$ cannot lie in $\D$.  Moreover, were it to lie on $\partial \D$, so that $1= |\bar{a}/\bar{b}| = |\bar{b}/\bar{a}|$, then $q$ would have Maclaurin series 
$$
q(z) = \frac{\bar{a}\bar{d}-\bar{b}\bar{c}}{\bar{a}} \sum_{n=0}^\infty (-1)^n\left( \frac{\bar{b}}{\bar{a}}\right)^{2n}z^n
$$
and $\sum_{n=0}^\infty \left| \frac{\bar{b}}{\bar{a}}\right|^{2n} = \infty$, so that  $q$ would not belong to $H^2$, a contradiction. 
Thus, $-\bar{a}/\bar{b}$ must lie outside the closed unit disk.  Hence $\left| \frac{\bar{b}}{\bar{a}}\right| = \left| \frac{b}{a}\right| < 1$, so that the zero $-\frac{b}{a}$ of $\phi$ lies in $\D$, as desired.
\end{proof}

Because the range of every composition operator contains the constant function $1$, the preceding proposition, combined with part (2) of Theorem~\ref{posinormal}, yields the following:

\begin{cor}\label{zeroC} If $C_\phi$ is posinormal on $H^2$, then $\phi(\beta) = 0$ for some $\beta \in \D$.\end{cor}

When $\phi$ is of parabolic type or dilation type with second fixed point on $\partial \D$, we show (see Theorems \ref{paraThm} and \ref{int_brd} below) that $\phi(\beta) = 0$ for some point in $\D$ is also sufficient for posinormality.  We remark that the stronger assumption that $C_\phi$ is hyponormal on $H^2$ yields the stronger conclusion that $\phi(0) = 0$; see \cite[Theorem 2]{cowenkrieteSubN}.    Theorem 5 of \cite{Cow} provides a characterization, which we discuss in Section 6 of this paper, of all hyponormal composition operators on $H^2$ having linear-fractional symbol.  Our work shows that the hyponormal composition operators on $H^2$ with linear-fractional symbols are precisely those that are posinormal with symbol vanishing at $0$; see Theorem~\ref{posi_hypo}.

As one would expect, characterizing posinormality for a composition operator having constant symbol is easy.

 \begin{proposition}\label{ConstantSymbol}  Suppose that $\phi$ is a constant selfmap of $\D$. Then the following are equivalent:
  \begin{itemize}  
  \item $C_\phi$ is posinormal;
  \item $C_\phi$ is coposinormal;
  \item $C_\phi$ is self-adjoint;
  \item $\phi=0$.
  \end{itemize}
  \end{proposition}
  \begin{proof} Let $\phi$ be a constant selfmap of $\D$, so that $\phi(z) = b$ for some $b$ in $\D$ and all $z\in \D$.    As we noted in the proof of Proposition~\ref{zero}, $C_\phi^* = K_b C_\sigma$, where $\sigma$ is the constant function taking value $0$ on $\D$. 
  
   If $C_\phi$ is posinormal, then  by Corollary~\ref{zeroC},  $\phi$ must vanish at a point in $\D$ and thus $b = 0$; i.e., $\phi$ is the zero function.  If $b = 0$, then note that $C_\phi^* = C_\phi$, so that $C_\phi$ is self-adjoint, hence normal and thus both posinormal and coposinormal.  Finally, if $C_\phi$ is coposinormal, then the range of $C_\phi^*$ must be contained in the range of $C_\phi$; in particular, we must have
$$
C_\phi^* 1 = K_{\phi(0)}  = K_b
$$
in the range of $C_\phi$. However, because the range of $C_\phi$ is the set of constant functions,  $K_{b}$ must be a constant function, so that $b = 0$.
\end{proof}

{\it For the remainder of this paper, unless otherwise indicated, we assume that all composition operators considered have nonconstant symbols.}

 Our main theorem is as follows:

\begin{theorem}[Main Theorem]
Suppose \ph\ is a (nonconstant)  linear-fractional selfmap of the unit disk \D\ with Denjoy-Wolff point $w$. \C\ is posinormal if and only if one of the following is true:

\begin{itemize}
    \item \ph\ is an automorphism of \D\ (so that \C\ is invertible),
    \item $\ph(z) = \tau \circ \alpha \tau$ where $|\alpha| < 1$, $\tau(z) = \frac{w-z}{1-\overline{w} z}$ for some $w \in \D$, and  $|w| < |\alpha|$,
    \item $|w| = 1, \ph'(w)=1$, and $\ph(\beta)=0$ for some $\beta \in \D$,
    \item \ph\ has one fixed point in \D\ and one on $\partial \D$, and $\ph(\beta) = 0$ for some $\beta \in \D$.
\end{itemize}

Furthermore, \C\ is coposinormal if and only if one of the following is true: 

\begin{itemize}
    \item  \ph\ is an automorphism \D\ (so that \C\ is invertible),
    \item the Denjoy-Wolff point $w$ of \ph\ is on $\partial \D$,
    \item $\ph(z) = \tau \circ \alpha \tau$ where $0< |\alpha| < 1$ and $\tau(z) = \frac{w-z}{1-\overline{w} z}$ for some $w \in \D$.
\end{itemize}
\end{theorem}

Since every invertible operator is posinormal \cite[Theorem 3.1]{Rhaly}, we have no work to do if \ph\ is an automorphism of \D.  For nonautomorphisms, we develop, in the next section, general requirements for posinormality and coposinormality of a composition operator with symbol $\phi$ and adjoint $\sigma$:
\begin{itemize}
\item  $C_\phi$ is posinormal if and only if $\phi\circ \sigma^{-1}$ is a selfmap of $\D$ and $\phi$ vanishes at a point in $\D$, and
\item $C_\phi$ is coposinormal if and only if $\sigma\circ \phi^{-1}$ is a selfmap of $\D$.
\end{itemize}
The remainder of our work is based on an analysis of how the fixed-point properties of $\phi$ determine when $\phi\circ\sigma^{-1}$ and $\sigma\circ\phi^{-1}$ are selfmaps of $\D$.   In Sections 3 and 4, we characterize posinormality and coposinormality when the Denjoy-Wolff point of $w$ lies on $\partial \D$. In Sections 5 and 6, respectively, we complete our characterization work, handling, in Section 5 the case where $\phi$ has Denjoy-Wolff point in $\D$ and its second fixed point outside of $\D^-$ and then in Section 6 the case where $\phi$ has Denjoy-Wolff point in $\D$ and its second fixed point  on the unit circle.   We conclude with questions for further research in Section 7. 

  The most challenging fixed-point scenario for our characterization work is that of Section 5, in which $\phi$ is assumed to have Denjoy-Wolff point in $\D$ and its second fixed point outside of $\D^-$.  To facilitate our work, we show that although in general posinormality and coposinormality are not similarity invariant, these properties are similarity invariant for composition operators with linear-fractional symbol when the similarity is induced by an automorphic composition operator (and, for posinormality of $C_\phi$, we assume $\phi$ vanishes a point of $\D$).  Also, in Section 5, we point out a connection between posinormality and complex symmetry that may be of independent interest:  Suppose $A$ is a bounded operator on a Hilbert space that is complex symmetric; then $A$ is posinormal if and only if it is coposinormal.

Observe that an operator $A$ that is both posinormal and coposinormal shares with normal operators the property that $\ran\, A = \ran\, A^*$ (applying the range inclusion condition (2)\ of Theorem~\ref{posinormal}).  It's natural to consider what other properties operators that are both posinormal and coposinormal share with normal operators.  For example, if $A$ is both posinormal and coposinormal must all powers of $A$ be posinormal? In Section 4, we establish that if $\phi$ is a parabolic nonautomorphism of $\D$ that vanishes at a point of $\D$, then $C_\phi$ is both posinormal and coposinormal; however, if $n$ is a sufficiently large positive integer, then $C_\phi^n$ is not posinormal.  Remark: This contradicts  Corollary 1(b) of \cite{Kubrusly}, which states that for any bounded linear operator on a Hilbert space $\mathcal{H}$, if $T$ is posinormal and coposinormal, then $T^n$ is posinormal and coposinormal for every $n\ge 1$.  

\section{Posinormality of $C_\phi$ and the Relationship Between $\phi$ and its adjoint $\sigma$}

In this section, we characterize posinormality and coposinormality in terms of the relationship between $\phi$ and its adjoint $\sigma$. 

\begin{theorem}\label{phisigmainv} Let $\phi(z) = \dfrac{az + b}{cz + d}$ be a (nonconstant) linear-fractional selfmap of $\D$, so that
$$
C_\phi^* = T_g C_\sigma T_h^*,
$$
where $ g(z) = \dfrac{1}{-\bar{b}z + \bar{d}}$ and $h(z) = cz+d$ are bounded analytic functions on $\D$ and $\sigma(z) = \dfrac{\bar{a}z - \bar{c}}{-\bar{b} z + \bar{d}}$ is a selfmap of $\D$.
\begin{itemize}
\item[(a)] If  $C_\phi: \Ht \rightarrow \Ht$ is posinormal  then the linear-fractional function
$$
\phi\circ \sigma^{-1}
$$
is an analytic selfmap of $\D$.  
\item[(b)]   If $\phi$ vanishes at a point of $\D$ and $\phi\circ \sigma^{-1}$ is a selfmap of $\D$, then $C_\phi$ is posinormal.
\end{itemize}
\end{theorem}

\begin{proof}   Note that  $\sigma$ is nonconstant because $\phi$ is nonconstant (in fact, $\phi'(z) = \frac{ad-bc}{(cz + d)^2}$ is nonzero if and only if $\sigma'(z) = \frac{\bar{a}\bar{d} - \bar{b}\bar{c}}{(-\bar{b}z + \bar{d})^2}$ is nonzero). Also note that
$$
\sigma^{-1}(z) = \frac{\bar{d}z +\bar{c}}{\bar{b}z + \bar{a}}.
$$
 Assume that  $C_\phi: \Ht \rightarrow \Ht$ is posinormal. Recall from Proposition \ref{zero}, this means $\phi$ must vanish at some point in $\D$, so that we must have $-b/a\in \D$. Thus, $-\bar{a}/\bar{b}$ lies outside the closure  of $\D$ and we see from the formula of $\sigma^{-1}$ above,  that $\sigma^{-1}$ is analytic on $\D$.
 
     Let $\beta$ be an arbitrary complex number lying outside the closure of $\D$.  Let $ f(z) = \dfrac{1}{z-\beta}$ and observe that $f$, being bounded and analytic on $\D$, belongs to $\Ht$.  Because $C_\phi$ is posinormal, the range of $C_\phi$ is a subset of the range of $C_\phi^*$. Thus, there is a function $q\in \Ht$ such that
$$
f\circ\phi = T_g C_\sigma T_h^*q.
$$
The preceding equation yields, for each $z\in \sigma(\D)$, 
$$
(-\bar{b} \sigma^{-1}(z) + \bar{d}) f(\phi(\sigma^{-1}(z))) = (T_h^* q)(z); \text{equivalently}, \frac{\bar{a}\bar{d}-\bar{b}\bar{c}}{\bar{b}z + \bar{a}}\left( \frac{1}{\phi(\sigma^{-1}(z))- \beta }  \right) =  (T_h^* q)(z).
$$
Observe that because $-\bar{a}/\bar{b}$ lies outside the closure of $\D$
$$
p(z):= \frac{\bar{a}\bar{d}-\bar{b}\bar{c}}{\bar{b}z + \bar{a}} \left( \frac{1}{\phi(\sigma^{-1}(z))- \beta }   \right)
$$
will fail to be analytic on $\D$ if and only if the linear-fractional transformation $\phi\circ\sigma^{-1}$ assumes the value $\beta$ at some $z_0$ in $\D$, in which case the function $p$ is analytic on $\D\setminus \{z_0\}$, having a pole of order 1 at $z_0$.  However, because  $T_h^* q$ agrees with $p$ on the nonempty open subset $\sigma(\D)$ of $\D$, the $\Ht$ function $T_h^* q$ provides an analytic continuation of $p$ to $\D$, preventing $p$ from having a pole at any point of $\D$.  Thus,  $\phi\circ\sigma^{-1}$ does not assume on $\D$ the value $\beta$. Because $\beta$ in the complement of the closure of $\D$ is arbitrary, we conclude that the linear-fractional transformation $\phi\circ\sigma^{-1}$ must map $\D$ into the closure of $\D$, but being nonconstant and analytic, it's an open map and we see $(\phi\circ\sigma^{-1})(\D)\subseteq \D$, completing the proof of part (a).  

Proof of part (b):   Now suppose that $\phi$ vanishes at some point of $\D$ and $\phi\circ \sigma^{-1}$ is an analytic selfmap of $\D$.  We complete the proof of part (b) by showing the range of $C_\phi$ is contained in the range of $C_\phi^*$.   Let $f\in \Ht$ be arbitrary.  
 
  Note that $f\circ (\phi\circ \sigma^{-1})$ belongs to $\Ht$ because we are assuming $\phi\circ \sigma^{-1}$ is an analytic selfmap of $\D$.   Note that $\left(\dfrac{1}{h}\right)(z) = \dfrac{1}{c z + d}$ is a bounded analytic function on $\D$ because $\phi(z) = \dfrac{az + b}{cz + d}$ is bounded  on $\D$.  Thus, $T_h^*$ has inverse $T_{1/h}^*$.   Just as in the proof of part (a), because $\phi$ vanishes at $-b/a \in \D$, $-\bar{a}/\bar{b}$ lies outside the closure of $\D$ and thus
  $$
  \left(\frac{1}{g\circ\sigma^{-1}}\right)(z) =\frac{\bar{a}\bar{d}-\bar{b}\bar{c}}{\bar{b}z + \bar{a}}
  $$
  is a bounded analytic function on $\D$.  Observe that
     $$
  q := T_{1/h}^*\, \frac{1}{g\circ\sigma^{-1}}\,  f\circ(\phi\circ\sigma^{-1})
  $$
belongs to $\Ht$ because $\frac{1}{g\circ\sigma^{-1}}\,  f\circ(\phi\circ\sigma^{-1})\in \Ht$, being the product of a function bounded and analytic on $\D$ and the $\Ht$ function $f\circ(\phi\circ\sigma^{-1})$.  It's easy to check that $C_\phi^* q = T_g C_\sigma T_h^* q = f\circ \phi$. It follows that $\text{ran}\ C_\phi \subseteq \text{ran}\ C_\phi^*$, as desired.
\end{proof}

Observe that when the hypotheses of part (b) of the preceding theorem hold,  the proof of (b) shows that $C_\phi = C_\phi^* T$, where $T = T_{1/h}^*\, T_{1/g\circ\sigma^{-1}}\, C_{\phi\circ\sigma^{-1}}$.   

 The following result characterizing coposinormality is an analogue of Theorem~\ref{phisigmainv}.  In its statement $\overline{\ph(\D)}$ denotes the set of complex conjugates of elements of $\ph(\D)$. 

\begin{theorem}\label{sigmaphiinv} Let $\phi$ be a linear-fractional selfmap of $\D$, and let $\sigma$ be its adjoint selfmap.  The following are equivalent:
\begin{itemize}
\item[(a)] $C_\phi: \Ht \rightarrow \Ht$ is coposinormal;
\item[(b)]  $1/\overline{\ph(\D)} \cap \ph^{-1}(\D) = \emptyset$;
\item[(c)]  $\sigma\circ\phi^{-1}$ is a selfmap of $\D$.
\end{itemize}
\end{theorem}

\begin{proof}  $(a)\implies(b)$: If \C\ is coposinormal, then ran $\Cs \subseteq$ ran \C.  Recall that $\Cs K_{\alpha} = K_{\phi(\alpha)}$ for $\alpha \in \D$; thus, for each $\alpha$ there must exist $f_{\alpha} \in \Ht$ so that $K_{\phi(\alpha)} = f_\alpha \circ \phi$.  Suppose, in order to obtain a contradiction,  that $1/\overline{\ph(\D)}$ and $\ph^{-1}(\D)$ intersect;  then there exists $\beta \in \D$ so that for some $\alpha$, $\overline{\ph(\alpha)}\ph^{-1}(\beta)=1$.   For each $z$ in the nonempty open subset $\phi(\D)$ of $\D$, we have
$$
f_{\alpha}(z) = K_{\phi(\alpha)} \circ \ph^{-1} (z) = \frac{1}{\vstrutb 1-\overline{\phi(\alpha)}\phi^{-1}(z)},
$$  
and thus, by analyticity, $f_\alpha$ and the linear-fractional function $q(z):= \dfrac{1}{\vstrutb 1-\overline{\phi(\alpha)}\phi^{-1}(z)}$ must agree on all of $\D$.  But $q$ has a pole at $\beta$, a contradiction.  Hence, $(a)\implies(b)$. \\

 $(b)\implies(c)$: Suppose that $1/\overline{\ph(\D)} \cap \ph^{-1}(\D) = \emptyset$.  Let $\E = \CP_\infty\setminus \D^-$ be the exterior of the closed disk.   Because $\phi$ and $\phi^{-1}$ are nonconstant analytic functions, they are open mappings. Thus,  $1/\overline{\ph(\D)} \cap \ph^{-1}(\D) = \emptyset$ is equivalent to 
 $$
 1/\overline{\ph(\D)}\subseteq \phi^{-1}(\E), 
 $$
 which implies 
 \begin{equation}\label{ExtInc}
 \phi\left(1/\overline{\ph(\D)}\right)\subseteq \E.
 \end{equation} 
 Let $\phi_e(z) = 1/\overline{\phi(1/\bar{z})}$. Because $\phi_e$ an analytic selfmap of $\E$, its inverse $\phi_e^{-1}$ must be a selfmap of $\D$.  In fact,  a computation, which may be based on 
 $$
 \phi_e^{-1} = \frac{1}{\overline{\phi^{-1}\left(\frac{1}{\bar{z}}\right)}},
 $$
   shows that $\phi_e^{-1} = \sigma$, the adjoint of $\phi$ (see \cite[Lemma 1]{Cow}).  Now observe that the inclusion (\ref{ExtInc}) shows that
 $\phi\circ \phi_e$ is an analytic selfmap of $\E$. 
  It follows that its inverse,
 $$
 (\phi\circ \phi_e)^{-1} = \phi_e^{-1}\circ \phi^{-1} = \sigma\circ \phi^{-1}
 $$
 is an analytic selfmap of $\D$, which completes the proof that $(b) \implies (c)$.
 
  $(c)\implies(a)$:  Suppose that $\sigma\circ \phi^{-1}$ is a selfmap of $\D$. Letting $\phi(z) = (az + b)/(cz +d)$ and assuming $b$ is nonzero, we have
  $$
  (\sigma\circ \phi^{-1})(z) = \frac{\bar{a}\phi^{-1}(z) - \bar{c}}{-\bar{b}\phi^{-1}(z) + \bar{d}}.
  $$
  Observe that $\sigma\circ\phi^{-1}$ has a non-removable singularity at $\phi(\bar{d}/\bar{b})$: 
  $$
  (\sigma\circ\phi^{-1})\left(\phi\left(\frac{\bar{d}}{\bar{b}}\right)\right) = \frac{(\bar{a}\bar{d} - \bar{b}\bar{c})/\bar{b}}{0},
  $$
  and $(\bar{a}\bar{d} - \bar{b}\bar{c})/\bar{b}$ is nonzero because $\phi$ is nonconstant.  Thus, $\phi(\bar{d}/\bar{b})$ lies outside of $\D^-$, which means, in particular, that
 \begin{equation}\label{gphiinvbdd}
  z\mapsto \frac{1}{-\bar{b}\phi^{-1}(z) - \bar{d}}  \ \text{is a bounded analytic function on}\ \D.
  \end{equation}
  Observe that if $\bar{b}$ is zero, the preceding function is still bounded and analytic on $\D$ (being a constant function).
    Let $g = 1/(-\bar{b}z + \bar{d})$, $h(z) = cz + d$, and 
  $$
  T = T_{g\circ \phi^{-1}} C_{\sigma\circ \phi^{-1}} T_{h}^*.
  $$
  Observe that $T$ is a bounded operator on $H^2$, being a product of three bounded operators (with the first factor being bounded by (\ref{gphiinvbdd})).  Now note that 
  $$
  C_\phi^* = C_\phi T,
  $$
  which shows, by part (4) of Theorem~\ref{posinormal}, that $C_\phi$ is coposinormal on $H^2$. This completes the proof that $(c)\implies (a)$ and thus the proof of the theorem.
  \end{proof}

Theorems~\ref{phisigmainv} and \ref{sigmaphiinv} above reduce our posinormality and coposinormality work to determining, respectively, when the linear-fractional transformations $\phi\circ\sigma^{-1}$ and $\sigma\circ\phi^{-1}$   are selfmaps of $\D$.   The level of difficulty of making these determinations varies significantly with the fixed-point properties corresponding to  $\phi$'s type---hyperbolic, parabolic, or dilation.

\section{The Hyperbolic Case}

  \bigskip
   
      Among the linear-fractional selfmaps of $\D$ having hyperbolic type, only the automorphisms induce posinormal composition operators. 
      
      \begin{theorem}\label{HCnonposi}
Suppose that $\phi$ is a non-automorphic linear-fractional selfmap of $\D$ having Denjoy-Wolff point $w$  belonging to $\partial\D$ such that $\phi'(w) <1$.  Then $C_\phi:H^2\rightarrow H^2$ is not posinormal.

\end{theorem}

\begin{proof}

Observe that $\psi(z) := \bar{w}\phi(wz)$ is a hyperbolic-type analytic selfmap of $\D$ having Denjoy-Wolff point $1$.  Moreover $C_\phi$ is unitarily equivalent to $C_\psi$ because $U:= C_{wz}$ is unitary (having adjoint $C_{\bar{w}z}$) and $C_\phi = U^* C_\psi U$.  Because posinormality is a unitary invariant, we may assume, without loss of generality, that $w = 1$, so that $\phi(1) = 1$ and $\phi'(1) < 1$.

To show that $C_\phi$ is not posinormal, we apply Theorem~\ref{phisigmainv} above---we establish that $\phi\circ\sigma^{-1}$ is not a selfmap of $\D$, where $\sigma$ is the adjoint of $\phi$.  The failure of $\phi\circ \sigma^{-1}$ to be a selfmap of $\D$ is readily apparent if we consider the natural selfmaps of the right halfplane corresponding to $\phi$ and $\sigma$. 

Let   $\nu(z) = \frac{1+z}{1-z}$, which maps $\D$ onto the right halfplane $\mathbb{\HP}$, taking $1$ to $\infty$. Thus, $\Phi := \nu \circ \phi \circ \nu^{-1}$ is a selfmap of the right halfplane  fixing $\infty$, and $\Phi(z) = sz + r$ for certain constants $s$ and $r$.  Note that $s > 0$ and $\Re(r)\ge 0$ because $\Phi$ maps the right halfplane to itself---in fact, because $\Phi$ is is not an automorphism, $\Re(r)>0$ . Since, for each $z\in \D$, the iterate sequence $\phi_n(z)$ converges to $1$, it follows that $|\Phi_n(z)| \rightarrow \infty$ as $n\rightarrow \infty$.  Thus, $s \ge 1$.  In fact, $s = \frac{1}{\phi'(1)} >1$:  if $t = \nu(x)$, then
$$
\phi'(1) = \lim_{x\rightarrow 1^-}    \frac{1-\phi(x)}{1-x}  = \lim_{t\rightarrow \infty} \frac{1-\nu^{-1}\Phi(t) }{1-\nu^{-1}(t)} =  \lim_{t\rightarrow \infty} \frac{1+t}{st+r + 1} = \frac{1}{s}.
$$

      The right halfplane incarnation of $\sigma$, $\nu\circ\sigma\circ\nu^{-1}$, is given by $\Sigma(z) = \frac{1}{s}z + \frac{\bar{r}}{s}$, which can be seen by noting that
  $$
  \phi(z) = \nu^{-1}(s\nu(z) + r) = \frac{ (-1 + r - s)z + 1- r-s}{(1+r-s) z -1-r-s}\   \text{and}\   \sigma(z) = \frac{ (-1 + \bar{r} - s)z - (1+\bar{r}-s)}{-(1-\bar{r}-s)z -1-\bar{r}-s}.
  $$
 We have $\Sigma^{-1}(z) =     sz - \bar{r}$ and $(\Phi\circ\Sigma^{-1})(z) = s^2z -s\bar{r} + r$  is not a selfmap of the right halfplane because $s > 1$ means $\Re(-s\bar{r} + r) = (1-s)\Re(r) <  0$.  Let $\beta$ be a number in $\HP$ such that $(\Phi\circ \Sigma^{-1})(\beta)$  has negative real part.  Observe that $z_0 = \nu^{-1}(\beta)$ is a point in $\D$ and
 $$
( \phi\circ\sigma^{-1})(z_0) = \nu^{-1}\left((\Phi\circ \Sigma^{-1})(\beta)\right)
 $$
 is not in $\D$ because $\nu^{-1}$ maps  $\HP$ onto $\D$ and $(\Phi\circ \Sigma^{-1})(\beta)$ has negative real part. Thus, the linear-fractional function $\phi\circ\sigma^{-1}$ is not a selfmap of $\D$ and $C_\phi$ cannot be posinormal.\end{proof}

On the other hand, selfmaps of hyperbolic type always induce coposinormal composition operators.  In his dissertation \cite{Sadraoui}, Sadraoui incidentally proved this, though the terminology  ``coposinormal''  is not used.   Sadraoui establishes that if $\phi$ is of  hyperbolic type, then there is a bounded linear operator $C$ on $H^2$ such that $C_\phi = C C_\phi^*$, equivalently, that $C_\phi^* = C_\phi C^*$, so that $C_\phi$ is coposinormal by via Rhaly's result Theorem~\ref{posinormal}, Part (4).   Because Sadraoui's dissertation was never published, we include a proof here, establishing coposinormality via Theorem~\ref{sigmaphiinv}, showing $\sigma\circ\phi^{-1}$ is a selfmap of $\D$ when $\phi$ is of hyperbolic type and $\sigma$ is its adjoint.  

\begin{theorem}\cite{Sadraoui}\label{dl1coposi}
Suppose \ph\ is a linear-fractional selfmap of $\D$ having Denjoy-Wolff point $w\in \partial \D$ with $\phi'(w) < 1$. Then $C_\phi:H^2\rightarrow H^2$ is coposinormal.
\end{theorem}

\begin{proof}
Just as in the proof of Theorem~\ref{HCnonposi}, we may assume $w = 1$, and, with  $\nu(z) = \frac{1+z}{1-z}$, we have, corresponding to $\phi$ and its adjoint $\sigma$, the following selfmaps of the right halfplane $\HP$:   $\Phi(z) =  (\nu \circ \phi \circ \nu^{-1})(z) = sz + r$ and $\Sigma(z) = \nu\circ\sigma\circ\nu^{-1}= \frac{1}{s}z + \frac{\bar{r}}{s}$, where $s > 1$ and $\Re(r) \ge 0$. Observe that
$$
(\Sigma \circ \Phi^{-1})(z) = \frac{1}{s}\left(\frac{z-r}{s}\right) +\frac{\bar{r}}{s} = \frac{1}{s^2}z + \frac{1}{s^2}(s\bar{r} - r)
     $$
    maps the right halfplane $\HP$ to itself---if $\Re(z) > 0$, then
    $$
    \Re\left( \frac{1}{s^2}z + \frac{1}{s^2}(s\bar{r} - r) \right) = \frac{1}{s^2} \Re(z) + \frac{1}{s^2}(s-1)\Re(r) > 0,   
    $$
    because $s > 1$ and $\Re(r)\ge 0$.  It follows that $\sigma\circ \phi^{-1}$ is a selfmap of $\D$:
    $$
    (\sigma\circ\phi^{-1})(\D) =  (\nu^{-1}\circ \Sigma \circ \Phi \circ \nu)(\D) = \nu^{-1}\left(\vstrut\Sigma(\Phi(\HP))\right) \subseteq \nu^{-1}(\HP) \subseteq \D.
    $$
    Thus, by Theorem~\ref{sigmaphiinv}, $C_\phi: H^2\rightarrow H^2$ is coposinormal.
    \end{proof}

 The preceding theorem shows all hyperbolic-type selfmaps of $\D$ give rise to coposinormal composition operators. We remark that some hyperbolic-type selfmaps give rise to cohyponormal composition operators---in fact cosubnormal ones.  Cowen and Kriete \cite{cowenkrieteSubN} have shown that if $\phi$ is of hyperbolic type with Denjoy-Wolff point $w$ and its second fixed point has the form $-tw$ for some $t\in [1, \infty]$, then $C_\phi$ is cosubnormal.

\section{The Parabolic Case}

We  continue to assume that the Denjoy-Wolff  point of $\phi$ is at $1$, and in the parabolic case, we have $\ph'(1)=1$. The right halfplane incarnation of such a mapping $\phi$ is given by
$$
\Phi_t(z): = (\nu\circ\phi\circ\nu^{-1})(z) = z + t,
$$
where $t$ is nonzero and $\Re(t) \ge 0$ (and $\nu(z) = (1+z)/(1-z)$, as in the preceding section).   Thus, these parabolic selfmaps of $\D$  form the following one-parameter family of functions:
\begin{equation}\label{Pfam}
  \quad \quad  \ph_t(z) = (\nu^{-1}\circ\Phi_t\circ \nu)(z) =  \frac{(2-t)z+t}{-tz+2+t}, \Re(t)\ge 0, t\ne 0.
\end{equation}
 Note that $\ph_t$ is an automorphism of $\D$ precisely when $\Re(t) = 0$ (because $\Phi_t$ is an automorphism of $\HP$ precisely when $\Re(t) = 0$). Also, note that whenever $t$ and $s$ are complex numbers with nonnegative real part, then 
\begin{equation}\label{SG}
 \ph_s\circ \phi_t =   \nu^{-1}\circ\Phi_s\circ \Phi_t \nu = \nu^{-1}\circ \Phi_{s + t} \circ \nu = \phi_{s + t}.
 \end{equation}
 Thus $\{\ph_t: \Re(t) \ge 0\}$ forms a semigroup all of whose members are selfmaps of $\D$ of parabolic type except for the identify function $\ph_0$.  
 
Assume $\Re(t) \ge 0$.  Observe that the adjoint $\sigma_t$ of $\ph_t$ is $\phi_{\bar{t}}$. Thus, $\phi_{t -\bar{t}}\circ\sigma_t =   \phi_{t -\bar{t}}\circ\phi_{\bar{t}}  = \phi_t$; that is,  $\phi_{2i{\tiny \Im(t)}}\circ\sigma_t = \phi_t$. Composing both sides of the preceding equation with $\sigma_t^{-1}$, we see that  $\phi_t\circ\sigma_t^{-1} = \phi_{2i{\tiny \Im(t)}}$, which is a selfmap of $\D$ (in fact, an automorphism of $\D$).   
 
 \begin{proposition}\label{d1posi} The selfmap $\ph_t$  of parabolic type induces a posinormal composition operator on $H^2$ if and only if $\ph_t$ vanishes at some point in $\D$.
\end{proposition}
\begin{proof} We have shown that for all members of the parabolic family $\ph_t$ given by (\ref{Pfam}), $\ph_t\circ \sigma_t^{-1}$ is a selfmap of $\D$. Thus, this proposition immediately from Proposition~\ref{zero} and Theorem~\ref{phisigmainv}.
\end{proof}  
 
 On the other hand, all members of the parabolic family (\ref{Pfam}) induce coposinormal composition operators.
 
 \begin{proposition} \label{coposinormald1}
The selfmap $\ph_t$  of parabolic type induces a coposinormal composition operator on $H^2$.  \end{proposition}

\begin{proof} Assume $\Re(t) \ge 0$.  We have
$$
\phi_{\bar{t} -t} \phi_t = \phi_{\bar{t}} = \sigma_t;
$$
that is, $ \sigma_t =  \phi_{-2i{\tiny \Im(t)}} \circ \phi_t$. Hence, $\sigma_t\circ \phi_t^{-1} =  \phi_{-2i{\tiny \Im(t)}}$, a selfmap of $\D$, and the proposition follows from Theorem~\ref{sigmaphiinv}.
\end{proof}

Combining the two preceding propositions, we see that if $\phi_t$ vanishes at some point of $\D$, then $C_{\ph_t}$ is both posinormal and coposinormal.  It's easy to check that $\phi_t$ vanishes at a point of $\D$ if and only if  $0 \le \Re(t) < 1$.  {\it Thus, in particular,  if  $ 0 <  \Re(t) < 1$, then $C_{\ph_t}$ is both posinormal and coposinormal, but not every power of $C_{\ph_t}$  is posinormal, since $C_{\ph_t}^{n} = C_{\ph_{nt}}$ and eventually we have $\Re(nt) > 1$, so that $C_{\phi_{nt}}$ is not posinormal because $\phi_{nt}$ does not vanish on $\D$.  }   Corollary 1(b) of \cite{Kubrusly} states that for any bounded linear operator on a Hilbert space $\mathcal{H}$, if $T$ is posinormal and coposinormal, then $T^n$ is posinormal and coposinormal for every $n\ge 1$.   However, as we have seen, any composition operator on $H^2$ induced by $\phi_t$ with $0 < \Re(t) < 1$, provides a counter-example to this Corollary.   

We have focused on parabolic selfmaps $\ph_t$ of $\D$ having Denjoy-Wolff point $1$ without loss of generality.  

\begin{theorem}\label{paraThm}
Suppose \ph\ is a linear-fractional selfmap of \D\ with a fixed point $w$ on $\partial \D$, and $\ph'(w) = 1$. Then \C\ is coposinormal on $H^2$. Furthermore, \C\ is posinormal on $H^2$ if and only if $\ph(\beta) =0$ for some $\beta \in \D$.
\end{theorem}

\begin{proof}  If $U:= C_{wz}$, then $U$ is unitary and $UC_\phi U^* = C_\psi$, where $\psi$ a is parabolic-type selfmap of $\D$ having Denjoy-Wolff point $1$, so that $\psi = \ph_t$ for some $t\ne0$ with $\Re(t)\ge 0$.  Being unitarily equivalent to $C_{\ph_t}$ for some $t$, we see that \C\ is coposinormal by  Proposition \ref{coposinormald1}.

We know that from Proposition \ref{d1posi} that $C_{\phi_t}$ is posinormal if and only if $\ph_t(\beta) = 0$ for some $\beta \in \D$. Since $\phi_t(z) = \psi(z) = \bar{w}\ph(wz)$ has a zero in \D\ if and only if \ph\ does, we see that \C\ is posinormal if and only if \ph\ has a zero in \D. 
\end{proof}

\section{Dilation Type I: Interior and Exterior Fixed Points }

We now turn our attention to composition operators $C_\phi$ induced by linear-fractional selfmaps $\phi$ of $\D$ having dilation type  (fixed point in $\D$) whose second fixed point lies outside of the closure $\D^-$ of $\D$.  This collection of composition operators includes those that are normal.  Recall that \C\ is normal on \Ht\ if and only if $\ph(z) = \alpha z$ for some $|\alpha| \leq 1$ (see, e.g.,\cite[Theorem 8.2]{cowen1995composition}).   In this section, we show that if $\phi$ is of dilation type and has its second fixed point outside of $\D^-$, then in order for $C_\phi$ to be posinormal or coposinormal $C_\phi$ must be similar to a normal composition operator via an automorphic composition operator; that is, $C_\phi = C_\tau^{-1}C_{\alpha z} C_\tau$, where $\alpha \in \D^-$ and $\tau$ is a conformal automorphism of $\D$ (and hence, $C_\tau^{-1} =C_{\tau^{-1}}$).  For the dilation-type symbols considered in this section, this similarity is both necessary and sufficient (with the additional condition for posinormality of having the symbol vanish at a point of $\D$): 
 \begin{itemize}
\item for $\phi$ vanishing at a point of $\D$,  $C_\phi$ is posinormal if and only if for some conformal automorphism $\tau$ of $\D$ and some $\alpha\in \D^-$, we have $C_\phi = (C_\tau)^{-1}C_{\alpha z} C_\tau$ is similar to the normal operator $C_{\alpha z}$;
\item   $C_\phi$ is coposinormal if and only if for some conformal automorphism $\tau$ of $\D$ and some $\alpha\in \D^-$, we have $C_\phi = (C_\tau)^{-1}C_{\alpha z} C_\tau$ is similar to the normal operator $C_{\alpha z}$.
\end{itemize}
Note that the operator equation 
$$
C_\phi = (C_\tau)^{-1}C_{\alpha z} C_\tau \ \text{is equivalent to the functional equation} \ \phi = \tau\circ \alpha \tau^{-1}.
$$
 In general, posinormality and coposinormality are not similarity invariant properties.  A  simple example is furnished by the following matrix operators acting on $\CP^2$:
$$
N= \begin{pmatrix} 1 & 0 \\ 0 & 0\end{pmatrix}, S = \begin{pmatrix} 1 & 0 \\-1 & 1\end{pmatrix}, \text{and}\  M = \begin{pmatrix} 1 & 0 \\  1 & 0\end{pmatrix}.
$$
Note $M = S^{-1}NS$ is similar to the normal operator $N$; however, the range of $M$ is the span of the vector $\begin{pmatrix} 1\\1\end{pmatrix}$ while the range of $M^*$ is the span of   $\begin{pmatrix} 1\\0\end{pmatrix}$.  Thus the range of $M$ is neither a subset nor a superset of the range of $M^*$, and we see $M$ is neither posinormal nor coposinormal.

\begin{lemma}\label{adjauto}  Let $\tau$ be a conformal automorphism of $\D$.   If the linear-fractional selfmap $\phi$ of $\D$ has adjoint $\sigma$, then the selfmap $\tau\circ\phi\circ \tau^{-1}$ had adjoint $\tau\circ\sigma\circ \tau^{-1}$.  
\end{lemma}
\begin{proof}  We know that $\tau$ has the form
\begin{equation}\label{taueq}
\tau(z) =  \lambda \frac{p - z}{1-\bar{p} z}, \text{where} \ p\in \D, \text{and} \ \lambda \ \text{is a unimodular constant}.
\end{equation}
Note $\tau^{-1}(z) = (p-\bar{\lambda} z)/(1-\bar{p} \bar{\lambda} z) = \bar{\lambda} (\lambda p - z)/(1-\overline{\lambda p} z)$.

 Observe that 
  \begin{equation}\label{autobs}
  (a)\ \  \tau\left(\frac{1}{\bar{z}}\right) = \frac{1}{\overline{\tau(z)}} \quad \text{and} \quad (b) \ \  \tau^{-1}\left(\frac{1}{\bar{z}}\right) = \frac{1}{\overline{\tau^{-1}(z)}}.
\end{equation}
  Then, employing the characterization of the adjoint $\sigma$ of $\phi$ of found in  \cite{Cow} and used in the proof of Theorem~\ref{sigmaphiinv}, namely,
  $$
  \sigma(z) = \frac{1}{\overline{\phi^{-1}\left(\frac{1}{\bar{z}}\right)}},
  $$
  we have that the  adjoint $\sigma_1$  of $\tau\circ\phi\circ\tau^{-1}$ is given by
   \begin{align*}
    \sigma_1(z)  = \frac{1}{\overline{(\tau\circ \phi\circ\tau^{-1})^{-1}\left(\frac{1}{\bar{z}}\right)}}  & =  \frac{1}{\overline{(\tau\circ \phi^{-1}\circ\tau^{-1})\left(\frac{1}{\bar{z}}\right)}}\\
      & =  \frac{1}{\overline{(\tau\circ \phi^{-1})\left(\tau^{-1}\left(\frac{1}{\bar{z}}\right)\right)}}\\
          & =  \frac{1}{\overline{\tau \left(\phi^{-1}\left(\frac{1}{\overline{\tau^{-1}(z)})}\right)\right)}} \quad (\text{by \ref{autobs} (b)})\\
            & = \tau\left(\frac{1}{\overline{ \phi^{-1}\left(\frac{1}{\overline{\tau^{-1}(z)})}\right)} }\right)\quad (\text{by \ref{autobs} (a)})\\
      & = (\tau\circ\sigma\circ\tau^{-1})(z),
    \end{align*}
    as desired.  
    \end{proof}

    We remark that a computationally intense proof of the preceding lemma may be constructed by
  \begin{itemize}
  \item[(i)]  letting $\phi(z) = (az + d)/(cz + d)$, so that $\sigma(z) = (\bar{a}z-\bar{c})/(-\bar{b}z + \bar{d})$ by Cowen's adjoint formula (see Theorem~\ref{adjformula});
  \item[(ii)] computing $\tau\circ\phi\circ \tau^{-1}$ using (\ref{taueq}) and the formula for $\phi$ from (i); 
  \item[(iii)] computing $\tau\circ\sigma\circ \tau^{-1}$ using (\ref{taueq}) and the formula for $\sigma$ from (i);
  \item[(iv)] verifying that applying Cowen's adjoint formula to the linear fractional map computed in (ii) yields the formula computed in (iii).
  \end{itemize}
  We now show that posinormality and coposinormality are preserved under similarities induced by automorphic composition operators (with the additional condition for posinormality of having the symbol vanish at a point of $\D$).
    
    \begin{proposition} \label{simProp} Let $\phi$ be a (nonconstant) linear-fractional selfmap of $\D$, and let $\tau$ be a conformal automorphism of $\D$.
    \begin{itemize}
    \item[(a)] Suppose that $C_\phi:H^2\rightarrow H^2$  is coposinormal and that $C_\psi$ is similar to $C_\phi$ via $C_\tau$: $C_\psi = (C_\tau)^{-1} C_\phi C_\tau$. Then $C_\psi$ is also coposinormal on $H^2$.
    \item[(b)]  Suppose that $C_\phi:H^2\rightarrow H^2$  is posinormal and that $C_\psi$ is similar to $C_\phi$ via $C_\tau$: $C_\psi =  (C_\tau)^{-1} C_\phi C_\tau$.  If $\psi$ vanishes at a point of $\D$, then $C_\psi$ is also posinormal on $H^2$.
    \end{itemize}
    \end{proposition}
    
\begin{proof}  Suppose that $C_\phi$  is coposinormal, that $C_\psi = (C_\tau)^{-1} C_\phi C_\tau$, and that $\sigma$ is the adjoint of $\phi$ while $\sigma_1$ is the adjoint of $\psi$.   Note $\psi = \tau\circ\phi\circ\tau^{-1}$, so that  by Lemma~\ref{adjauto}, we have $\sigma_1 =  \tau\circ\sigma\circ\tau^{-1}$. 

Because $C_\phi$ is coposinormal, we know that $\sigma\circ\phi^{-1}$ is a selfmap of $\D$ by Theorem~\ref{sigmaphiinv}.  Hence, because $\tau$ and $\tau^{-1}$ map $\D$ onto $\D$, we have
$$
\tau\circ( \sigma\circ\phi^{-1})\circ\tau^{-1}  \ \text{is a selfmap of}\ \D.
$$
Because  $\tau\circ( \sigma\circ\phi^{-1})\circ\tau^{-1}= \sigma_1\circ\psi^{-1}$, we see $C_\psi$ is coposinormal by Theorem~\ref{sigmaphiinv}.  

 Now suppose that  $C_\phi$  is posinormal, that $C_\psi = (C_\tau)^{-1} C_\phi C_\tau$, that $\psi$ vanishes at a point of $\D$,  and that $\sigma$ is the adjoint of $\phi$ while $\sigma_1$ is the adjoint of $\psi$.   Just as in the proof of part (a) of this proposition, we have $\psi = \tau\circ\phi\circ\tau^{-1}$, so that  by Lemma~\ref{adjauto}, we have $\sigma_1 =  \tau\circ\sigma\circ\tau^{-1}$.  Because $C_\phi$ is posinormal, we have from Theorem~\ref{phisigmainv}, part (a),  that $\phi\circ\sigma^{-1}$ is a selfmap of $\D$.  Hence, because $\tau$ and $\tau^{-1}$ map $\D$ onto $\D$, we have
$$
\tau\circ( \phi\circ\sigma^{-1})\circ\tau^{-1}  \ \text{is a selfmap of}\ \D.
$$
Because  $\tau\circ( \phi\circ\sigma^{-1})\circ\tau^{-1}= \psi\circ\sigma_1^{-1}$, and $\psi$ vanishes at some point in $\D$, we see $C_\psi$ is posinormal by Theorem~\ref{phisigmainv}, part (b).  
\end{proof}

  Keep in mind that we are assuming our linear-fractional symbols to be nonconstant.  Observe that the characterization of posinormality provided in the next theorem continues to be valid, by Proposition~\ref{ConstantSymbol}, if we assume $\phi$ is constant, except we have $w = \alpha = 0$ (so that $|w| = |\alpha|)$; also, if $|\alpha| = 1$, then $C_{\tau\circ\alpha\tau}$ will be invertible hence posinormal (but in this case $\phi=\tau\circ\alpha\tau$ is an elliptic automorphism so that it would inaccurate to describe its fixed point $w\in \D$ as its Denjoy-Wolff point).   
    \begin{theorem}\label{S4posi}
Suppose  that \ph\ is a linear-fractional selfmap of $\D$ with its Denjoy-Wolff point $w$  in \D\ and its other fixed point outside of $\D^-$. Then $C_\phi:H^2\rightarrow H^2$ is posinormal if and only if
\begin{quotation}
$\ph(z) = \tau \circ \alpha \tau$, where $|\alpha| < 1$, $\tau = \frac{w-z}{1-\overline{w}z}$, and $|w| < |\alpha|$. 
\end{quotation}
\end{theorem}

  \begin{proof}   Note that $\tau$ is a self-inverse automorphism of $\D$. Suppose  $\ph(z) = \tau \circ \alpha \tau$, where $|\alpha| < 1$, $\tau = \frac{w-z}{1-\overline{w}z}$, and $|w| < |\alpha|$.  Then $\phi$ vanishes at the point $\tau(w/\alpha)\in \D$ and $C_\phi$ is similar to the normal composition operator $C_{\alpha z}$ via the automorphic composition operator $C_\tau$: $C_\phi  = (C_\tau)^{-1} C_{\alpha z} C_\tau =C_\tau C_{\alpha z} C_\tau$. Thus, $C_\phi$ is posinormal by Proposition~\ref{simProp}, Part (b).  
  
  Conversely, suppose that $C_\phi$ is posinormal and that $\sigma$ is the adjoint of $\phi$.  Let $\tau = \frac{w-z}{1-\overline{w}z}$ (where $w$ is the Denjoy-Wolff point of $\phi$).  Let  $C_\psi= (C_\tau)^{-1} C_\phi C_\tau =  (C_\tau C_\phi C_\tau)$, so that $\psi = \tau\circ\phi\circ \tau$ and $C_\psi$ is similar to $C_\phi$ via the automorphic composition operator $C_\tau$.  Because $\tau$ maps $0$ to the Denjoy-Wolff point of $\phi$, we see $\psi$ has Denjoy-Wolff point $0$.  Hence,
  $$
  \psi(z) = \frac{az}{1-cz}
  $$
  for some constants $a$ and $c$.   Note $\psi'(0) = a$ so that $|a| < 1$ because $0$ is the Denjoy-Wolfff point of $\psi$ (or apply Schwarz's Lemma).
  Because, $\psi = \tau\circ\phi\circ \tau$ vanishes at a point of $\D$ (namely $0$) and $C_\phi$ is posinormal, we see from Proposition~\ref{simProp} that $C_\psi$ is posinormal.    We claim that the posinormality of $C_\psi$ yields $c = 0$, so that $\psi(z) = az$, where $a\in \D$.  Thus, if our claim is valid, then because $\psi = \tau\circ\phi\circ \tau$, we have $az =  \tau\circ\phi\circ \tau$; equivalently, that $\phi = \tau \circ a \tau$  and since $\phi$ vanishes at  point $\beta\in \D$ (because $C_\phi$ is posinormal), we must have $0=\tau(a\tau(\beta))$, so that $|w| = |a\tau(\beta)| = |a||\tau(\beta)| < |a|$.  Thus, our proof will be complete if we can establish our claim that $c$ must be 0.  
  
  Suppose, in order to obtain a contradiction, that $c \ne 0$ and $\psi(z) = az/(1-cz)$ induces a posinormal composition operator on \Ht.   
   A computation shows that the exterior fixed point $p$ of $\phi$ is given by $p = (1-a)/c$.  Let $\tau_1$ to be the self-inverse automorphism of \D\ given by
 $$
 \tau_1(z) = \frac{\overline{\frac{1}{p}} - z}{1- \frac{1}{p} z} = \frac{\frac{\bar{c}}{1-\bar{a}} - z}{1- \frac{c}{1-a} z}.
 $$
 A computation shows that
\begin{equation}\label{TE1}
\tpsi(z): = (\tau_1\circ \psi\circ \tau_1)(z) = a z + \bar{c} \, \eta,
\end{equation}  
where $\eta$ is the unimodular constant $(1-a)/(1-\bar{a})$.     Note that for $\psi(z) = az/(1-cz)$, $C_\psi^* = C_{\sigma_1} T_h^*$, where $\sigma_1(z) = \bar{a}z + \bar{c}$ and $h(z) = 1 - c z$.  Also note that
\begin{equation}\label{TE2}
\mu(z): = (\tau_1\circ \sigma_1^{-1} \circ \tau_1)(z)  = \frac{z}{\rule{0in}{0.12in}\bar{a} +c\, \bar{\eta} z}.
\end{equation}

 We are assuming that $C_\psi$ is posinormal and $c \ne 0$. Because it's posinormal, by Theorem \ref{phisigmainv} we conclude $\psi\circ \sigma_1^{-1}$ must be a selfmap of \D.  We have
$$
(\psi\circ\sigma_1^{-1})(\D) \subseteq \D;
$$
hence,
$$
(\tau_1 \circ \psi \circ \tau_1 \circ \tau_1\circ  \sigma_1^{-1}\circ \tau_1)(\D)  \subseteq \D \quad \text{so that via Equations (\ref{TE1}) and (\ref{TE2}})\quad  (\tpsi\circ \mu)(\D) \subseteq  \D.
$$
Using the definition of $\tpsi$, we see that the inclusion on the right above yields
\begin{equation}\label{cis01}
  a \mu(\D) +\bar{c}\, \eta \subseteq \D.
\end{equation}
Recall that $\mu(z) = \frac{z}{\rule{0in}{.12in} \bar{a} + c \, \bar{\eta} z}$.  If $|c\, \bar{\eta}| \ge |\bar{a}|$, equivalently, $|c| \ge |a|$ (recall we are assuming $c \ne 0$), then $\mu$ is unbounded on $\D$ and the inclusion $(\ref{cis01})$ cannot hold.  

Now suppose that $|c| < |a|$.  We claim the image of $\D$ under $\mu$ is a disk of radius $\frac{|a|}{|a|^2 -|c|^2} > \frac{1}{|a|}$ (because we are assuming $c \ne 0$ and $|c| < |a|$). Recall $|a| < 1$.   Hence, if our claim is valid, then $a\mu(\D)$ is a disk having radius exceeding $1$, and,  being a translate of a disk whose radius exceeds $1$,
$$
 a \mu(\D) +\bar{c}\, \eta \ \text{cannot be contained in}\ \D,
$$
which contradicts (\ref{cis01}) above.  This contradiction shows our proof is complete  provided we establish the claim above regarding the radius of the disk $\mu(\D)$, which we now do.  

The linear-fractional transformation $\mu(z) = \frac{z}{\rule{0in}{0.12in}\bar{a} +c\, \bar{\eta} z}$ maps the line through the origin and $\frac{\bar{a}}{\rule{0in}{0.12in} c\, \bar{\eta}}$ to the line $L$ through the origin and $\eta/c$.   The circle $\mu({\partial \D})$ intersects the line $L$ (at right angles) at  
$$
\mu\left(\frac{\eta \bar{a}|c|}{|a|c}\right) =   \frac{\eta|c|/c}{|a| + |c|} \ \text{and at}\  \mu\left(-\frac{\eta \bar{a}|c|}{|a|c}\right) =  \frac{-\eta|c|/c}{|a| - |c|}.
$$
Thus $\mu(\D)$ is a circle of radius
$$
\frac{1}{2}\left|\frac{\eta|c|}{c}\left(  \frac{1}{|a| + |c|} \ +  \frac{1}{|a| - |c|}\right)\right| = \frac{|a|}{|a|^2 - |c|^2},
$$
as claimed, completing the proof.

  \end{proof} 
  
 Regarding the following theorem, note that for all $|\alpha|\le 1$, $C_{\tau\circ \alpha\tau}$ will be coposinormal but the cases $\alpha = 0$ and $|\alpha| = 1$  are not consistent, respectively, with our standing assumption that $\phi$ is nonconstant and with the hypothesis that $\phi$ has a Denjoy-Wolff point in $\D$.  

\begin{theorem}\label{S4coposi}
Suppose  that \ph\ is a linear-fractional selfmap of $\D$ with Denjoy-Wolff point $w$ inside \D\ and its other fixed point outside of $\D^-$. Then $C_\phi:H^2\rightarrow H^2$ is coposinormal if and only if
\begin{quotation}
$\ph(z) = \tau \circ \alpha \tau$, where $|\alpha| < 1$ and $\tau = \dfrac{w-z}{1-\overline{w}z}$.
\end{quotation}

\end{theorem}

\begin{proof}   

Just as in the proof of Theorem~\ref{S4posi}, if $\ph(z)$ takes the form  $\tau \circ \alpha \tau$, where $0< |\alpha| < 1$ and $\tau = \frac{w-z}{1-\overline{w}z}$ is automorphism of $\D$, then $C_\phi$ is similar to the normal composition operator $C_{\alpha z}$ via the automorphic composition operator $C_\tau$, and thus, $C_\phi$ is coposinormal by Proposition~\ref{simProp}, Part (a).  
  
   Conversely, suppose that $\phi$ is coposinormal and that $\sigma$ is the adjoint of $\phi$. Again, we proceed as in the proof of Theorem~\ref{S4posi}.  Let $\tau = \frac{w-z}{1-\overline{w}z}$, and  let  $C_\psi= (C_\tau)^{-1} C_\phi C_\tau =  (C_\tau C_\phi C_\tau)$, so that $\psi = \tau\circ\phi\circ \tau$.   Because $\tau$ maps $0$ to the Denjoy-Wolff point of $\phi$, we see $\psi$ has Denjoy-Wolff point $0$.  Hence,
  $$
  \psi(z) = \frac{az}{1-cz}
  $$
  for some constants $a$ and $c$ (and $a\ne 0$ because we are assuming $\phi$ to be nonconstant).      Because, $C_\psi$ is similar to $C_\phi$ via the automorphic composition operator $C_\tau$ and $C_\phi$ is  coposinormal, we see from Proposition~\ref{simProp} that $C_\psi$ is coposinormal.   By Theorem~\ref{sigmaphiinv}, 
  $$
    \sigma_1\circ \psi^{-1} \ \text{must be a selfmap of \D},
    $$
where $\sigma_1(z) = \bar{a}z + \bar{c}$ is the adjoint of $\psi$.  Thus,
\begin{equation}\label{cis02}
(\sigma_1\circ \psi^{-1})(\D) = \bar{a} \psi^{-1}(\D) + \bar{c} \subseteq \D. 
\end{equation}
That the preceding inclusion can hold only if $c = 0$ follows just as in the proof of Theorem~\ref{S4posi}: again we must have $|a| < 1$,  and 
$$
\psi^{-1}(z) =  \frac{z}{a + cz}
$$
has the the same form as the function $\mu$ appearing in the proof of Theorem~\ref{S4posi}, and thus if $c\ne 0$, (\ref{cis02}) cannot hold if $|a| \le |c|$ because, in this case, $\psi^{-1}$ is unbounded on $\D$, and in case $|a| >|c|$, $\psi^{-1}(\D)$ is disk of radius exceeding $1/|a| > 1$ showing that (\ref{cis02}) cannot hold.   

 Because $c = 0$, we have $\psi(z) = az = \tau\circ\phi\circ \tau$, so that 
 $$
 \phi = \tau\circ a\tau
 $$
 for some $a\in \D\setminus\{0\}$, as desired.
 \end{proof}
   
   We conclude this section by pointing out a connection between our work and that of Narayan et al.\ in \cite{narayan2016complex}.  The selfmaps of $\D$ of the form $z\mapsto az/(1-cz)$ and $z\mapsto az +c$ playing prominent roles in the proofs of Theorems~\ref{S4posi} and Theorem~\ref{S4coposi} induce complex symmetric composition operators on $H^2$. 
An operator $A$ is \textit{complex symmetric} if there exists a conjugate-linear, isometric involution $J$ such that $JAJ = A^{*}$.  It was shown in \cite[Main Theorem]{narayan2016complex} that if $\ph(z)=\frac{az}{1-cz}$ or $\ph(z)=az+c$ are dilation-type having second fixed point of outside of $\D^-$, then $C_\phi$ is complex symmetric.    The following general result may be of independent interest.
\begin{proposition}\label{cso}
Suppose $A$ is a bounded operator on a Hilbert space that is complex symmetric. Then $A$ is posinormal if and only if it is coposinormal.  \end{proposition}

\begin{proof}
 Suppose that $C_\phi$ is posinormal.  By Theorem \ref{posinormal}, there exists a bounded operator $T$ so that $A = A^{*}T$. By complex symmetry, there exists a conjugate-linear, isometric involution $J$ so that $AJ = JA^{*}$. Then, $A^{*}= JAJ = JA^{*}TJ = AJTJ$; that is, $A^{*} = AJTJ$.  Note that the presence of two factors of the conjugate-linear operator $J$, makes $T_1:=JAJ$ a linear operator such that $A^* = AT_1$; so, by Theorem \ref{posinormal}, $A^{*}$ is posinormal, i.e., $A$ is coposinormal. We just showed that the adjoint of a posinormal operator is posinormal; thus, if $A^*$ is posinormal, then $A=(A^*)^*$ is posinormal; i.e., $A$ coposinormal implies $A$ is posinormal.  
 \end{proof}
 
 We have shown that if the complex symmetric composition operator $C_{az/(1-cz)}$ or $C_{az + c}$ is posinormal  then it's actually normal (i.e., $c$ must be $0$).

\section{Dilation Type II: Interior and Boundary Fixed Points }

Our final scenario to consider is when the Denjoy-Wolff point $w$ of \ph\ belongs to \D, but the second fixed point of \ph\ is on the boundary of \D.  These symbols are adjoints  of the hyperbolic-type selfmaps of Section 3.  Since hyperbolic selfmaps induced coposinormal composition operators, it is not too surprising that adjoints of these maps induce posinormal composition operators provided we assume the symbol vanishes at some point in $\D$. 

\begin{theorem}\label{int_brd}
Suppose \ph\ is a linear-fractional selfmap of \D\ with one fixed point in \D\ and one on $\partial \D$. Then $C_\phi:H^2\rightarrow H^2$ is posinormal if and only if $\ph(\beta) =0 $ for some $\beta \in \D$.
\end{theorem}

\begin{proof}
By Proposition \ref{zero}, if  \C\ is posinormal, then $\ph$ must vanish at a point of $\D$. 

 If $\phi$ is an automorphism, then $C_\phi$ is invertible, so that $C_\phi$ is posinormal. 
 
Suppose that $\phi$, not an automorphism, vanishes at some $\beta\in \D$.  We again assume, without loss of generality, that the fixed point of \ph\ on $\partial \D$ is $1$.  Just as in the proof of Theorem~\ref{HCnonposi}, the right  halfplane incarnation of $\phi$, $\nu\circ\phi\circ \nu^{-1}$, where $\nu = (1+z)/(1-z)$ is given by
$$
\Phi(z) = sz + r,
$$
but now $0 \le s < 1$ and $\Re(r) > 0$ (strict inequality because we are assuming $\phi$ is not an automorphism).  If $\sigma$ is the adjoint of $\phi$, then just as before,
$$
 \Sigma(z) = \frac{1}{s}z + \frac{\bar{r}}{s} \quad \text{and} \quad \Sigma^{-1}(z) =     sz - \bar{r},
$$
where $\Sigma = \nu\circ\sigma\circ \nu^{-1}$.
Observe that 
$$
(\Phi\circ \Sigma^{-1})(z) = s^2z -s\bar{r} + r
$$
 is a selfmap of the right halfplane since $\Re(-s\bar{r} + r) = (1-s)\Re(r) > 0$ because $s < 1$.
It follows that $\phi\circ\sigma^{-1}$ is a selfmap of $\D$. Since $\phi$ is assumed to vanish at a point of $\D$, we conclude that $C_\phi$ is posinormal by Theorem~\ref{phisigmainv}.
\end{proof}

In \cite{Cow}, Cowen characterizes all hyponormal composition operators on $H^2$ having linear-fractional symbol:  $C_\phi$ is hyponormal if and only if  either (i) $\phi(z) = \alpha z$ for some $\alpha \in \D^-$, so that $C_\phi$ is normal, or (ii) $\phi(z) = \frac{sz}{1-(1-s)\bar{\eta}z}$ where $0 < s< 1$ and $|\eta| = 1$, so that $\phi$ has Denjoy-Wolff point $0$ and it second fixed point $\eta$ on the unit circle.  Cowen also shows that the hyponormal composition operators induced by selfmaps of $\D$ given by (ii) are actually subnormal and hence subnormality is equivalent to hyponormality for composition operators having linear-fractional symbol.  Here we show hyponormality is equivalent to posinormality with symbol vanishing at $0$.
\begin{theorem}\label{posi_hypo}
Suppose \ph\ is a linear-fractional selfmap of \D\ (and we allow $\phi$ to be a constant function).  The following are equivalent:
\begin{itemize}
\item[(a)] $C_\phi$ is subnormal;
\item[(b)] $C_\phi$ is hyponormal;
\item[(c)] $C_\phi$ is posinormal and $\phi(0) = 0$;
\item[(d)]  either (i) $\phi(z) = \alpha z$ for some $\alpha \in \D^-$, so that $C_\phi$ is normal, or (ii) $\phi(z) = \dfrac{sz}{1-(1-s)\bar{\eta}z}$ where $0 < s< 1$ and $|\eta| = 1$, so that $\phi$ has Denjoy-Wolff point $0$ and it second fixed point $\eta$ lies on the unit circle. 
\end{itemize}
\end{theorem}

\begin{proof}  If $\phi$ is a constant function and $C_\phi$ is posinormal, we've seen (Proposition~\ref{ConstantSymbol}) that $\phi$ takes the form (i) of $(d)$ with $\alpha = 0$ so that $C_\phi$ is normal (in fact, self-adjoint) and all the equivalences of the theorem hold.  For the remainder of the proof, we assume $\phi$ is nonconstant.  

That $(a)\implies(b)$  is true in general.   Also hyponormality implies posinormality in general (with the identity operator being an interrupter).  Moreover, as we have noted, if $C_\phi: H^2\rightarrow H^2$ is hyponormal, then $\phi(0) = 0$ (\cite[Theorem 2]{cowenkrieteSubN}). Thus, we see $(b)\implies (c)$. 

Suppose $(c)$ holds.  Then, either $\phi$ is an elliptic automorphism: $\phi = \alpha z$ for some $\alpha\in \partial \D$ and is described by (i), or  $\ph$ has Denjoy-Wolff point $0$.   Suppose $\phi$ has Denjoy-Wolff point $0$.  If its second fixed point lies outside $\D^-$, then, because we are assuming $C_\phi$ is posinormal, Theorem~\ref{S4posi} tells us that $\phi = \tau\circ \alpha \tau$ where $\tau(z) = z$ so that $\phi(z) = \alpha z$ where $0 < |\alpha| < 1$; thus,  $\phi$ is, once again, described by (i).  Suppose that $\phi$'s second fixed point $\eta$ lies on the unit circle. Then $\psi(z) : = \bar{\eta}\phi(\eta z)$ fixes $1$ as well as $0$.  The righthalfplane incarnation $\Psi$ of $\psi$, $\nu\circ\psi\circ\nu^{-1}$ where $\nu(z) = (1+z)/(1-z)$, takes the form $\Phi(z) = sz + 1-s$ where $0 <s < 1$,  because $\Phi$ must fix $1$ (since $\phi$ fixes $0$).  Thus $\psi = \nu^{-1}\circ \Psi\circ \nu$ is given by
$$
\psi(z) = \frac{sz}{1-(1-s)z} \  \text{and}\  \phi(z) = \eta\psi(\bar{\eta}z) = \frac{sz}{1-(1-s)\bar{\eta}z} \ \text{is described by (ii)}.
$$
We have shown that $(c)\implies (d)$.

If $(d)$ holds and $\phi$ is described by (i), then $C_\phi$ is normal hence subnormal,  If (d) holds and $\phi$ is described by (ii), then \cite[Theorem 5]{Cow} establishes that $C_\phi$ is subnormal.  Thus $(d)\implies(a)$, which completes the proof. 
\end{proof}

Hyperbolic-type composition operators (treated in Section 3) were never posinormal; so, given we are now working with symbols that are adjoints of those of Section 3, we have reason to believe they should never induce coposinormal composition operators. That this is the case is a quick corollary of Theorem~\ref{sigmaphiinv}.

\begin{theorem}
Suppose \ph\ is a nonautomorphic linear-fractional selfmap of \D\ with one fixed point in \D\ and another on $\partial \D$. Then \C\ is not coposinormal on $H^2$.
\end{theorem}

\begin{proof} 
Just as in the proof of Theorem~\ref{int_brd}, the right  halfplane incarnation of $\phi$, $\nu^{-1}\circ\phi\circ \nu$, where $\nu = (1+z)/(1-z)$ is given by
$$
\Phi(z) = sz + r,
$$
where $0 < s < 1$ and $\Re(r) > 0$ (because we are assuming $\phi$ is not an automorphism).  Just as before, the right halfplane incarnation of the adjoint $\sigma$ of $\phi$ is given by $\Sigma(z) = \frac{1}{s}z + \frac{\bar{r}}{s}$.  Note $\Phi^{-1}(z) = \frac{1}{s}z - \frac{r}{s}$.  

Observe that 
$$
(\Sigma\circ\Phi^{-1})(z) = \frac{1}{s^2}-\frac{r}{s^2} + \frac{\bar{r}}{s}.
$$
 Thus $(\Sigma\circ\Phi^{-1})$ is not a selfmap of the right halfplane the since $\Re(-\frac{r}{s^2} + \frac{\bar{r}}{s}.) = \frac{\Re(r)}{s} \left(-\frac{1}{s}+1\right) < 0$ because $s < 1$.
It follows that $\sigma\circ\phi^{-1}$ is a not a selfmap of $\D$. Thus $C_\phi$ is not coposinormal by Theorem~\ref{sigmaphiinv}.

\end{proof}

The preceding theorem completes our characterization of posinormality and coposinormality for composition operators on $H^2$ having linear-fractional symbol.

\section{Further Questions}

The natural question to raise is to what extent can we characterize all composition operators on \Ht\ which are posinormal or coposinormal. Without access to simple analogs of Cowen's adjoint formula, we face a much more difficult task. General, necessary conditions for  coposinormality seem easier to obtain than those for posinormality. For example, it's easy to see that $C_\phi$ must have dense range if $C_\phi$ is coposinormal and $\phi$ is nonconstant (because in this case the range of $C_\phi$ must include the range of $C_\phi^*$, which, in turn, contains the reproducing kernels $K_\alpha$ for all $\alpha$ in the nonempty open subset $\phi(\D)$ of $\D$).  Note that when $C_\phi$ has dense range, $\phi$ is necessarily univalent. Also note that there are non-univalent, nonconstant symbols that induce posinormal composition operators; e.g., any inner function that vanishes at $0$ induces an isometric composition operator, and isometries are posinormal (in fact, subnormal). 

  One possible path forward would be to use the more complicated adjoint formulas for composition operators with rational symbols---see \cite{CowGal},  \cite{hammond2008adjoints}, and \cite{BrdShap}. Some work has also been done for irrational symbols, see \cite{gu2020adjoints}. 

In the introduction, we mentioned that there is little known about general spectral results for posinormal operators. It is possible that focusing on composition operators might suggest some results, but the posinormality-coposinormality story told in this paper suggests otherwise. The symbols determined to induce (co)posinormal composition operators in this paper give rise to a wide variety of spectral properties; see \cite[Chapter 7]{cowen1995composition}.  

\section*{Acknowledgement}

The second author would like to thank Taylor University for the sabbatical funding that led to the completion of this work. 


{\footnotesize{
    \bibliography{citations}
    \bibliographystyle{plain}
    }}

\end{document}